\documentclass[12pt]{article}

\usepackage{amssymb}
\usepackage{amsmath}
\usepackage{amsbsy}
\usepackage{amscd}
\usepackage{amsfonts}
\usepackage{amsthm}
\usepackage{mathrsfs}
\usepackage{verbatim}
\usepackage[colorlinks]{hyperref}
\usepackage{fullpage}
\usepackage{mathdots}
\usepackage{graphicx,subfigure}
\usepackage[english]{babel}
\usepackage[utf8]{inputenc}
\usepackage{tikz}
\usepackage{adjustbox}
\usepackage{authblk}
\usepackage{caption}

\usepackage{mathtext}

\usepackage{tikz-cd}
\usetikzlibrary{backgrounds,fit, matrix}
\usetikzlibrary{positioning}
\usetikzlibrary{calc,through,chains}
\usetikzlibrary{arrows,shapes,snakes,automata, petri}

\newtheorem{Theorem}[equation]{Theorem}
\newtheorem{Corollary}[equation]{Corollary}
\newtheorem{Lemma}[equation]{Lemma}
\newtheorem{Proposition}[equation]{Proposition}

\theoremstyle{definition}
\newtheorem{Definition}[equation]{Definition}
\newtheorem{Example}[equation]{Example}

\newtheorem{Notation}[equation]{Notation}

\newtheorem{Question}[equation]{Question}
\newtheorem{Remark}[equation]{Remark}

\numberwithin{equation}{section}
\numberwithin{figure}{section}

\newcommand{\PP}{{\mathbb P}}
\newcommand{\C}{{\mathbb C}}
\newcommand{\Z}{{\mathbb Z}}
\newcommand{\Q}{{\mathbb Q}}

\newcommand{\rs}{\mathcal{R}_{G}}
\newcommand{\bs}{\mathcal{B}_{G}}

\newcommand{\mc}[1]{\mathcal{#1}}

\newcommand{\beq}{\begin{equation}}
\newcommand{\eeq}{\end{equation}}

\begin{document}

\title{On the Borel Submonoid of a Symplectic Monoid}

\author{Mahir Bilen Can, Hayden Houser, Corey Wolfe}

\affil{\small{
Department of Mathematics, Tulane University,\\
6823 St. Charles Ave, New Orleans, LA, 70118,\\}} 

\maketitle

\begin{abstract}

In this article, we study the Bruhat-Chevalley-Renner order on the complex symplectic monoid $MSp_n$.
After showing that this order is completely determined by the Bruhat-Chevalley-Renner order on the 
linear algebraic monoid of $n\times n$ matrices $M_n$, we focus on the Borel submonoid of $MSp_n$.
By using this submonoid, we introduce a new set of type B set partitions.
We determine their count by using the ``folding'' and ``unfolding'' operators that we introduce. 
We show that the Borel submonoid of a rationally smooth reductive monoid with zero is rationally smooth. 
Finally, we analyze the nilpotent subsemigroups of the Borel semigroups of $M_n$ and $MSp_n$.
We show that, contrary to the case of $MSp_n$, the nilpotent subsemigroup of the Borel submonoid of $M_n$ is irreducible. 
\vspace{.5cm}


\end{abstract}

\normalsize

\section{Introduction}

Let $M$ be a complex reductive monoid with unit group $G$, and let $B$ be a Borel subgroup in $G$. 
Then we have a square of inclusions as in the following diagram  
\[
\begin{tikzcd}
  \overline{B}  \arrow[hookrightarrow]{r} & M \\
B \arrow[hookrightarrow]{u} \arrow[hookrightarrow]{r}  & G \arrow[hookrightarrow]{u} 
\end{tikzcd}
\]
where $\overline{B}$ is the Zariski closure of $B$ in $M$; we will call $\overline{B}$ the {\em Borel submonoid} determined by $B$. 
Although its combinatorics and geometry are relatively less explored compared to that of the ambient reductive monoid,  
the Borel submonoid is a very important object for the study of the representation theory of $M$~\cite[Theorem 3.4]{Doty99}.
In the special case of the linear algebraic monoid of $n\times n$ matrices, the $B\times B$-orbits 
in $\overline{B}$ are parametrized by the set partitions of $\{1,\dots, n\}$, 
providing a gateway to an unchartered domain for combinatorialists, see~\cite{CanCherniavsky}. 
In this regard, our goal in this paper is to present first combinatorial results, whose analogous versions are obtained in~\cite{CanCherniavsky}, for the Borel submonoid of a ``symplectic monoid'' that we will define next.

Let $l$ be a positive integer, and set $n:=2l$. 
The set of all $l\times l$ matrices with entries from $\C$ will be denoted by $M_l$. 
We let $J$ denote the $n\times n$ matrix, 
$
J=
\begin{bmatrix}
0 & J_l \\
-J_l & 0 
\end{bmatrix},
$
where $J_l$ is the unique antidiagonal $l\times l$ permutation matrix, that is, 
\[
J_l=
\begin{bmatrix}
0 & 0 & \cdots &0 & 1  \\
0 & 0 & \cdots &1 & 0  \\
\vdots & \vdots & \iddots& \vdots &\vdots  \\
0 & 1 & \cdots & 0 & 0  \\
1 & 0 & \cdots & 0 & 0  \\
\end{bmatrix}.
\]
The {\em symplectic group} is defined by $Sp_n := \{ A\in GL_n :\ A^\top J A = J \}$.
This is the group of linear automorphisms of $\C^n$ that preserve the skew-bilinear form that is defined by $J$. 
(Once we fix the even integer $n=2l$, in the sequel, it will be convenient for us to denote $Sp_n$ by $G$.)
Let us denote the central extension of $Sp_n$ in $GL_n$ by $GSp_n$. 
This is the smallest reductive subgroup of $GL_n$ that contains both of the subgroups $G$ and 
the group of invertible scalar matrices $\{ c I_n :\ c\in \C^*\}$, where $I_n$ denotes the $n\times n$ identity matrix. 
The Zariski closure of $GSp_n$ in $M_n$ is called the {\em $n$-th symplectic monoid}.
Such monoids were first considered by Grigor'ev~\cite{Grigorev1981}.
The following concrete description of the $n$-th symplectic monoid, which we will denote by $MSp_n$,
is due to Doty~\cite[Proposition 4.3]{Doty98}: $MSp_n:= \{ A\in M_n :\ A^\top J A=A J A^\top = cJ,\ c\in \C \}$.
Basic geometric ingredients (the Renner monoid, the cross section lattice, and a cell decomposition) of $MSp_n$ 
are described explicitly by Li and Renner in~\cite{LiRenner}. 
An in-depth analysis of the rational points of $MSp_n$ over finite fields, 
including some fascinating combinatorial formulas about its Renner monoids, 
are described by Cao, Lei, and Li in~\cite{CaoLeiLi}. 
To describe the main results of our paper, next, we will briefly review the Renner monoid of $MSp_n$ in relation 
with the rook monoid.

To keep our notation simple, let us denote by $B$ the Borel subgroup consisting of the upper triangular matrices in $GSp_n$.
The {\em natural action} of $B\times B$ on $MSp_n$ is defined by 
$(b_1,b_2) \cdot x  = b_1 x b_2^{-1}$, where $b_1,b_2\in B,\ x \in MSp_n$.
This action has finitely many orbits~\cite{Renner86, LiRenner} and moreover the orbits are parametrized by a finite inverse semigroup, 
$MSp_n = \bigsqcup_{\sigma \in \mathcal{R}_G} B \sigma B$. 
The finite inverse semigroup $\mathcal{R}_G$ is called the {\em symplectic Renner monoid};
it is the symplectic version of the {\em rook monoid} $\mathcal{R}_n$, which consists of 0/1 square matrices 
of size $n$ with at most one 1 in each row and each column. In fact, $\mathcal{R}_G$ is a submonoid of $\mathcal{R}_n$. 
The elements of $\mathcal{R}_n$ are called rooks, and we will call the elements of $\mc{R}_G$ the {\em symplectic rooks}.
The {\em Bruhat-Chevalley-Renner order} on $\mathcal{R}_n$ is defined by 
\begin{align}\label{A:BCR}
\sigma \leq \tau \iff  B_n \sigma B_n \subseteq  \overline{B_n \tau B_n}
\end{align}
for $\sigma,\tau \in \mathcal{R}_n$. 
(We will introduce the most general form of the Bruhat-Chevalley-Renner order in the preliminaries section.)
An explicit combinatorial description of $\leq$, in the spirit of Deodhar's criteria, is obtained in~\cite{CanRenner12}.
By using this explicit characterization of $\leq$, it is shown in~\cite{Can19:Shellable} that $(\mathcal{R}_n,\leq)$ is a graded, bounded, EL-shellable poset.

The first main observation in our paper, Theorem~\ref{T:inrsn}, states that, for $\sigma, \tau \in \mathcal{R}_G$, we have 
\begin{align*}
\sigma \leq \tau \text{ in } \mathcal{R}_G \iff  \sigma \leq \tau \text{ in } \mathcal{R}_n.
\end{align*}
An important family of subposets of $\mathcal{R}_n$ are defined as follows. Let $k$ be an integer in $\{0,1,\dots, n\}$, and let 
\[
\mathcal{B}_n(k):= \{ \sigma \in \mathcal{R}_n:
\text{$\sigma$ is upper triangular and rank$(\sigma)=k$}\} 
\]
and
\[
\mc{B}_n := \bigsqcup_{k=0}^n \mc{B}_n(k).
\]
Then $\mc{B}_n$ parametrizes the $B_n\times B_n$-orbits in the Borel monoid $\overline{B_n}$. 
Actually, $\mc{B}_n$ is the lower interval $[0,1]= \{ x\in \mathcal{R}_n :\ x\leq 1\}$ in $\mc{R}_n$. 
Therefore, $\mc{B}_n$ is also EL-shellable. 
Generalizing this observation, in~\cite{CanCherniavsky}, joint with Cherniavsky, 
the first author showed that each poset $(\mc{B}_n(k),\leq)$ ($k\in \{0,1,\dots, n\}$) is 
a graded, bounded, EL-shellable poset. 
In fact, it turns out that $(\mc{B}_n(k),\leq)$ is a union of ${n\choose k}$ maximal subintervals all of which have the same minimum element.
An important combinatorial aspect of this development is that, as a set, $\mc{B}_n(k)$ is in bijection with 
the set partitions of $\{1,\dots, n+1\}$ with $k$ blocks. In particular, the cardinality $|\mc{B}_n(k)|$ is given by the 
Stirling numbers of the second kind, $S(n+1,k)$.  
In our second main result, we obtain similar results for the rank $k$ elements of the Borel submonoid $\overline{B}$ in $MSp_n$. 
We should mention that the type BC analogs of the set partitions with respect to ``refinement order'' is well known~\cite{Reiner97}.
For a more recent study of their combinatorial properties, see~\cite{BagnoBiagioliGarber}.

We will denote by $\mc{B}_G$ the submonoid of all upper triangular elements in the symplectic Renner monoid $\mc{R}_G$. 
In other words, $\mc{B}_G = \{ x\in \mc{R}_G:\ x\leq 1\} = [0,1]$ in $\mc{R}_G$.  
The {\em $k$-th symplectic Stirling poset}, denoted by $\mc{B}_G(k)$, is the subposet defined by 
\begin{align}
\mc{B}_G(k) := \{ x \in \mc{B}_G :\ \text{rank} (x) = k\}.
\end{align}
In Theorem~\ref{T:maxelements}, we prove that the $k$-th symplectic Stirling poset is a graded bounded poset with unique minimum element,
and with ${l \choose k} 2^k$ maximal elements, all of which are rank $k$ diagonal idempotents. 
It is now a natural question to find the cardinality of each of the posets $\mc{B}_G(k)$. 
We answer this question in Theorem~\ref{T:formula}. 
It turns out that 
\begin{align*}
|\mc{B}_G(k) | = \sum_{a+b+c =k} 2^{a+c}3^b   {l \choose b} S(l+1, l+1-a) S(l+1, l+1-c),
\end{align*}
where $(a,b,c)\in \Z_{\geq 0}^3$. Here, for integers $s,t\in \Z$ such that $s\leq t$, 
$S(s,t)$ stands for the $(s,t)$-th Stirling numbers of the second kind.

Reductive monoids are regular in the semigroup theory sense. 
Geometrically, the only smooth reductive monoids with one-dimensional center
are the monoids of $n\times n$ matrices~\cite{Renner85}. 
A complex algebraic variety of dimension $n$ is called {\em rationally smooth} if for every $x\in X$, 
the local cohomology groups $H^i(X,X\setminus \{x\})$ are zero for $i\neq 2n$, and $H^{2n}(X,X\setminus \{x\}) = \Q$. 
It turns out that the rationally smooth reductive monoids have rich combinatorial and geometric structures~\cite{Renner08,RennerDescent}. 
Their classification has been completed by Renner~\cite{Renner08,Renner09}. 
In particular, $MSp_n$ is a rationally smooth monoid. 
Gonzales showed that the rationally smooth reductive monoids are GKM manifolds, see~\cite{Gonzales14}. 
This means that the relevant (equivariant) cohomological data for such a monoid can be recovered from the knowledge of 
torus invariant points and curves alone. 

In Theorem~\ref{T:rationallysmooth}, we show that, if $M$ is a rationally smooth reductive monoid with zero, 
and $\overline{B}$ is a Borel submonoid in $M$, then $\overline{B}$ is rationally smooth as well. 
Although we do not exploit this information here, we can now use Renner's $H$-polymomials for computing the intersection 
cohomology Poincar\'e polynomials of many Borel submonoids. 
In particular, this idea is applicable to the case of $MSp_n$. 
We plan to revisit this topic in a future paper.

We now describe the organization of our paper. 
In Section~\ref{S:Preliminaries} we briefly summarize some basic properties of the symplectic groups and monoids.
Section~\ref{S:BCR} is devoted to the study of the Bruhat-Chevalley-Renner order on $MSp_n$. 
In this section we prove our first result, Theorem~\ref{T:inrsn}.
Empowered by the concrete description of the partial order, 
we begin our study of the Borel submonoid of $MSp_n$ in Section~\ref{S:MSp_n}.
In particular, in this section, we give a count of the number of rank $k$ elements of the symplectic upper triangular rooks, Theorem~\ref{T:formula}. 
The purpose of Section~\ref{S:RationallySmooth} is to show that the Borel submonoids of 
rationally smooth reductive monoids with zeros are rationally smooth, Theorem~\ref{T:rationallysmooth}.
In the final part of our paper, we return to our study of the symplectic monoids. 
We observe that, unlike the case of the monoid of $n\times n$ matrices, 
the subsemigroup of nilpotent elements of the Borel submonoid of $MSp_n$ is not irreducible for $n\geq 2$. 
\\

\textbf{Acknowledgements.}
The authors thank Yonah Cherniavsky and Zhenheng Li.
The first author is partially supported by a grant from the Louisiana Board of Regents.

\section{Preliminaries}\label{S:Preliminaries}

In this section we will review the basic ingredients of our objects. 

\subsection{Symplectic groups.}
Let $n$ be a positive integer of the form $n=2l$ for some $l\in \Z$. 
Let us denote by $H$ (resp. $G$) the special linear group $SL_n$ (resp. the symplectic group $Sp_n$).
Then $G\subseteq H$ with equality if $n=2$. 
We will denote by $T_H$ and $B_H$ the maximal diagonal torus and the Borel subgroup consisting of upper triangular matrices in $H$, respectively. 
Then the intersections $T_G:=G\cap T_H$ and $B_G := G\cap B_H$ are, respectively, 
the maximal diagonal torus and a Borel subgroup containing $T_G$ in $G$. 

Let $\theta : H \to H$ denote the following involutory automorphism:
\[
\theta (A) = J (A^\top)^{-1} J^{-1} \qquad A\in H.
\]
Then the fixed subgroup of $\theta$ in $H$ is $G$. In other words, we have $H^\theta = G$. 
Also, it is easy to verify that $B_G = B_H^\theta$ and that $T_G = T_H^\theta$.
With this choice of $T_H$ and $B_H$, we know that the normalizer of $T_H$ in $H$, 
that is, $N_H(T_H)$ is equal to the $n\times n$ monomial matrices in $H$, 
and the elements of the Weyl group, $W_H:=N_H(T_H)/T_H$, are represented by the 
permutation matrices of size $n$. 
We will denote $W_H$ by $\mathcal{S}_{n}$. 
The {\em one-line notation} of an element $w$ of $S_n$ is the sequence $(w_1,\dots, w_n)$, 
where $w_i = w(i)$ for $i\in \{1,\dots, n\}$. 
In this notation, the Weyl group of $(G,T_G)$ has a convenient description as the fixed point 
subgroup of the induced involution, $\theta : S_n \to S_n$ which is defined by
\[
\theta(w) := ( n+1 - w_n, n+1 - w_{n-1},\dots, n+1 - w_1) \qquad w\in S_n.
\]
In other words, we have 
\[
W_G = \{ w\in S_n :\ \theta(w) = w \}.
\]
By working with the root system corresponding to $(G,B_G,T_G)$, 
one knows that $(W_G,S_G)$, as a Coxeter group, is generated by 
\[
S_G=\{ r_i r_{n-i} :\ 1\leq i \leq l-1 \} \cup \{ r_l \},
\]
where $r_j$ ($j\in \{1,\dots, n-1\}$) denotes the simple transposition $r_j = (j, j+1)$ in $S_n$. 
Let us define $s_1,\dots, s_l$ by setting
\begin{align}\label{A:s}
s_j :=
\begin{cases}
r_j r_{n-j} & \text{ if $j\in \{1,\dots, l-1\}$;}\\
r_l & \text{ if $j=l$.}\\
\end{cases}
\end{align}
In this notation, the Coxeter-Dynkin diagram of $(G,B_G,T_G)$ can be depicted as in Figure~\ref{F:Directed}.
This labeling is consistent with the labeling that is given in~\cite{Bourbaki456}.
\begin{figure}[htp]
\begin{center}
\scalebox{1}{
\begin{tikzpicture}
\node (r1) at (-2,0) {$s_1$};
\node (r2) at (0,0) {$s_2$};
\node (r3) at (2,0) {};
\node (r4) at (2.5,0) {$\dots$};
\node (r5) at (3,0) {};
\node (r6) at (5,0) {$s_{l-1}$};
\node (r7) at (7.25,0) {$s_{l}$};
\draw[-, line width=1pt] (r1) -- (r2);
\draw[-, line width=1pt] (r2) -- (r3);
\draw[-, line width=1pt] (r3) -- (r4);
\draw[-, line width=1pt] (r4) -- (r5);
\draw[-, line width=1pt] (r5) -- (r6);
\draw[-, line width=1.1pt]  (5.5,0.05) -- (7,0.05);
\draw[-, line width=1.1pt]  (5.5,-0.05) -- (7,-0.05);
\end{tikzpicture}
}
\caption{The Coxeter-Dynkin diagram of type $\textrm{C}_l$.}
\label{F:Directed}
\end{center}
\end{figure}

\subsection{Symplectic monoids.}

Let $M$ be a reductive monoid with unit group $G$. 
Then, by definition, $G$ is a connected reductive group.
Let $B$ be a Borel subgroup in $G$, and let $T$ be a maximal torus of $G$ that is contained in $B$. 
Then the Weyl group of $G$ is given by $W:=N_G(T)/T$. 
The Bruhat-Chevalley decomposition of $G$ is the finite decomposition $G= \bigsqcup_{w\in W} B \dot{w} B$. 
Likewise, the $B\times B$-orbits in $M$ are parametrized by a finite inverse semigroup, which is called the {\em Renner monoid} of $M$;
it is defined as the quotient $R:=\overline{N_G(T)}/T$, where $\overline{N_G(T)}$ denotes the Zariski closure of $N_G(T)$ in $M$. 
Then the {\em Bruhat-Chevalley-Renner decomposition} of $M$ is given by 
\[
M= \bigsqcup_{r\in R} B \dot{r} B.
\]
The dot on $r$ indicates that we are choosing a representative of $r$ from $\overline{N_G(T)}$. 
In general, it is not true that $R$ is a submonoid of $M$. 
An excellent survey of the Renner monoids of classical monoids can be found in~\cite{CLL_Fields}.

\begin{Notation}
Let $G$ denote, as before, the symplectic group $Sp_n$. Then 
the Renner monoid of the symplectic monoid $MSp_n$ will be denoted by $\rs$. 
The Weyl group of $G$ will be denoted by $W_G$. 
\end{Notation}

1) Let $\theta$ denote the involution that we introduced before, 
that is, $\theta(i) = n -i +1$ for $i\in \{1,\dots, n\}$. 
A subset $S\subseteq \{1,\dots, n\}$ is called an {\em admissible subset} if $\theta (S) \cap S =\emptyset$. 
For $i,j\in \{1,\dots, n\}$, let $E_{ij}$ denote the $(i,j)$-th elementary matrix.
Then the $(k,l)$-th entry of $E_{ij}$ is 1 if $(k,l)=(i,j)$, and it is 0 if $(k,l)\neq (i,j)$. 
In~\cite[Theorem 3.1.8]{LiRenner}, it is shown that 
\begin{align}
\rs= 	\left\{ \sum_{i\in I,w\in W_G} E_{i,wi} :\ I \text{ is admissible} \right\}.
\end{align}

2) 
An {\em injective partial transformation on $\{1,\dots, n\}$} is an injective map $f: D \to R$,  
where $D=D(f)$ and $R=R(f)$ are two subsets from $\{1,\dots, n\}$ with equal cardinalities. 
\begin{Definition}
The set of all injective partial transformations on $\{1,\dots, n\}$ is called the {\em rook monoid};
we will denote it by $\mathcal{R}_n$. As we mentioned before, $\mc{R}_n$ is the Renner monoid of $M_n$. 
\end{Definition}
In~\cite[Theorem 3.1.10]{LiRenner}, Li and Renner show that 
\begin{align}\label{A:LiRenner}
\rs=	\left\{x \in \mathcal{R}_n :\ \text{$D(x)$ and $R(x)$ are admissible, and $x$ is singular} \right\} \cup W_G.
\end{align}

3) 
For $k\in \{0,1,\dots, n\}$, let $e_k$ denote the diagonal idempotent, 
$e_k := E_{11} + E_{22}+\cdots + E_{kk}$. 
Also, let $e_0$ denote the $n\times n$ 0-matrix. 
The cross-section lattice of $\rs$ is then given by 
\begin{align}\label{A:crossG}
\Lambda_G := \{e_0,e_1,e_2,\dots,e_l, e_n\}.
\end{align}
(Notice the jump in the indices of the last two idempotents $e_l$ to $e_n$. This is not a typo!) 
In this notation, the Renner monoid of $MSp_n$ is given by 
\begin{align}
\rs = \bigsqcup_{e_i \in \Lambda } W_G e_i W_G.
\end{align}
Since $e_n$ is the identity element, the subset $W_Ge_nW_G$ is equal to $W_G$. 
Therefore, the rank of a singular element in $\rs$ is at most $l$.

\section{The Bruhat-Chevalley-Renner order}\label{S:BCR}

Let $M$ be a reductive monoid with unit group $G$. 
Let $B$ be a Borel subgroup in $G$, and let $T$ be a maximal torus in $G$ 
such that $T\subseteq B$. Let us denote the Renner monoid of $M$ by $R$. 
The {\em Bruhat-Chevalley-Renner order} on $R$ is the following partial order: for $x,y\in R$, 
\begin{align*}
x\leq_{BCR} y \iff BxB\subseteq \overline{ByB},
\end{align*}
where the bar over $ByB$ stands for the Zariski closure in $M$. 
Whenever it is clear from the context, we will omit writing the subscript $BCR$ in $\leq_{BCR}$.
Note that the restriction of $\leq$ to $W$ is known as the Bruhat-Chevalley order on $W$, which is defined by the same formulation, 
\begin{align*}
x\leq y \iff BxB\subseteq \overline{ByB}\ \text{ for $x,y\in W$},
\end{align*}
where the bar over $ByB$ stands for the Zariski closure in $G$.

The Weyl group $W$ is a graded poset with the rank function $\ell_W : W\to \Z_{\geq 0}$ defined by 
\[
\ell_W(w) = \dim BwB - \dim B\ \text{ for $w\in W$}.
\]
Note that $W$ is a Coxeter group and it has a system of Coxeter generators, denoted by $S$. 
For $w\in W$, $\ell_W(w)$ can also be defined as the minimal number of 
simple reflections $s_{i_1},\dots, s_{i_r}$ from $S$ with $w=s_{i_1}\cdots s_{i_r}$.
A subgroup that is generated by a subset $I\subset S$ will be denoted by $W_I$ and it will be called a 
{\it parabolic subgroup of $W$}. 
For $I\subseteq S$, we will denote by $D_I$ the following set:
\begin{align}
D_I:= \{ x\in W:\ \ell_W(xw ) = \ell_W(x) + \ell_W(w) \text{ for all } w\in W_I \}.
\end{align}

The {\em type-map}, $\lambda : \Lambda \to 2^S$, is defined by $\lambda(e) := \{s\in S:\ s  e= e s  \}$ for $e \in \Lambda$. 
The containment ordering between $G\times G$-orbit closures in $M$ 
is transferred via $\lambda$ to a sublattice of the Boolean lattice on $S$. 
Associated with $\lambda(e)$ are the following sets:
$\lambda_*(e) := \cap_{f\leq e} \lambda (f)$ and $\lambda^*(e):= \cap_{f\geq e} \lambda(f)$.
We define the subgroups 
\[
W(e):= W_{\lambda(e)},\qquad W_*(e) := W_{\lambda_*(e)},\qquad W^*(e) := W_{\lambda^*(e)}.
\]
Then we have 
\begin{enumerate}
\item $W(e)= \{ a\in W:\ ae=ea \}$,
\item $W^*(e)= \cap_{f\geq e} W(f)$,
\item $W_*(e)= \cap_{f\leq e} W(f) = \{ a\in W:\ ae= ea = e\}$.
\end{enumerate}
We know from~\cite[Chapter 10]{Putcha} that 
$W(e),W^*(e)$, and $W_*(e)$ are parabolic subgroups of $W$, and furthermore, we know that 
$W(e) \cong W^*(e) \times W_*(e)$. 
If $W(e) = W_I$ and $W_*(e) = W_K$ for some subsets $I,K\subset S$, then we define 
$D(e) := D_I$ and $D_*(e):= D_K$. 
\vspace{.25cm}

\textbf{Theorem/Definition (Pennell-Putcha-Renner):}
For every $x\in WeW$ there exist elements $a\in D_*(e), b\in D(e)$, which are uniquely determined by $x$, such that 
\begin{align}\label{A:std}
x= a e b^{-1}.
\end{align}
The decomposition of $x$ in (\ref{A:std}) will be called the {\em standard form of $x$}.
Let $e, f$ be two elements from $\Lambda$. It is proven in~\cite{PPR97} that 
if $x= a eb^{-1}$ and $y= cf d^{-1}$ are two elements in standard form in $R$, then 
\begin{align}\label{A:BCR}
x \leq y \iff e\leq f,\ a \leq cw,\ w^{-1} d^{-1} \leq b^{-1} \qquad \text{for some $w\in W(f)W(e)$.}
\end{align}
We will occasionally write $D(e)^{-1}$ to denote the set $\{ b^{-1} :\ b\in D(e)\}$.

\subsection{Deodhar's criteria.}

Our goal in this section is to present a practical description of the 
Bruhat-Chevalley-Renner order on the rook monoid.

Let us denote by $B_n$ the Borel subgroup of invertible upper triangular matrices in $M_n$.
The Renner monoid $\mathcal{R}_n$ is a graded poset with the following rank function~\cite{Renner86}:
\begin{equation*}
\ell( x )= \dim (B_n x B_n),\  x \in \mathcal{R}_n.
\end{equation*}
There is a combinatorial formula for computing the values of $\ell$~\cite{CanRenner12}.

We represent elements of $\mathcal{R}_n$ by $n$-tuples. 
For $x=(x_{ij}) \in \mathcal{R}_n$ we define the sequence $(x_1,\dots,x_n)$ by
\begin{equation}\label{E:oneline}
x_j = 
\begin{cases}
0  &\text{if the $j$-th column consists of zeros,}\\
i &\text{if $x_{ij}=1$.}
\end{cases}
\end{equation}
By abuse of notation, we denote both the matrix and the sequence $(x_1,\dots,x_n)$ by $x$.
For example, the associated sequence of the partial permutation matrix 
\begin{equation*}
x=\begin{bmatrix}  
0 & 0 & 0 & 0 \\
0 & 0 & 0 & 0 \\
1 & 0 & 0 & 0 \\ 
0 & 0 & 1 & 0
\end{bmatrix}
\end{equation*}
is $x=(3,0,4,0)$.

Next, we define a useful partial order on finite sets of integers. 
Let $\{i_1,\dots, i_k\}$ and $\{j_1,\dots, j_k\}$ be two equinumerous sets of integers such that 
$i_1<\cdots < i_k$ and $j_1 < \cdots < j_k$. 
We will write 
\begin{align}\label{A:partialorder}
\{i_1,\dots, i_k\} \leqslant \{j_1,\dots, j_k\}\iff i_1\leq j_1, i_2 \leq j_2, \dots, i_k \leq j_k.
\end{align}
Let $x=(x_1,\dots, x_n)$ be an element from $\mathcal{R}_n$. 
For $i\in \{1,\dots,n\}$, we define 
\[
\tilde{x}(i) := \{ x_1,\dots, x_i\}.
\]
In this notation, the main result of~\cite{CanRenner12} is as follows. 

\begin{Theorem}\label{T:CR}
Let $x=(x_1, \dots , x_n)$ and $y=(y_1, \dots , y_n)$ be two elements from $\mathcal{R}_n$. 
Then $x \leq_{BCR} y$ if and only if for every $i\in \{1,\dots, n-1\}$ we have $\tilde{x}(i) \leqslant \tilde{y}(i)$. 
\end{Theorem}

\begin{Example}
Let $x= (3,1,5,2,4)$ and $y= (5,2,4,3,1)$ be two elements from $\mathcal{R}_5$. 
Since 
\begin{align*}
\tilde{x} (1) = \{ 3\} &\leqslant \{ 5\} = \tilde{y} (1), \\ 
\tilde{x} (2) = \{ 1,3\} &\leqslant \{ 2,5\} = \tilde{y} (2), \\ 
\tilde{x} (3) = \{ 1,3,5\} &\leqslant \{ 2,4,5\} = \tilde{y} (3), \\ 
\tilde{x} (4) = \{ 1,2,3,5\} &\leqslant \{ 2,3,4,5\} = \tilde{y} (4),
\end{align*}
we see that $x\leq_{BCR} y$. 
\end{Example}

Theorem~\ref{T:CR} is quite useful for explicit computations.

\subsection{The Bruhat-Chevalley-Renner order on $\mc{R}_G$.}

As before, let $n$ be an even number, $n=2l$, $l\in \Z_+$. 
Recall that $W_G$ denotes the Weyl group of $Sp_n$. 
Then $W_G$ is the centralizer in $\mathcal{S}_n$ of the involution $\theta = (1,n) (2, n-1)\cdots (l,l+1)$.
As a Coxeter group, $W_G$ has type $\text{BC}_l$.
In~\cite[Corollary 8.1.9]{BjornerBrenti}, it is shown that for two elements $u$ and $v$ from $W_G$,
\begin{align}\label{A:inCn}
u \leq v\ \text{ in }\ W_G \iff u \leq v\ \text{ in }\ W_H=\mathcal{S}_n.
\end{align}
We will extend (\ref{A:inCn}) to the Renner monoid of $MSp_n$.

\begin{Theorem}\label{T:inrsn}
Let $x$ any $y$ be two elements from $\rs$. 
Then 
\begin{align*}
x \leq y\ \text{ in }\ \rs \iff x \leq y\ \text{ in }\ \mathcal{R}_n.
\end{align*}
\end{Theorem}
\begin{proof}
We write $x$ and $y$ in their standard form, $x= a eb^{-1}$ and $y= cf d^{-1}$, where $a\in D_*(e),b\in D(e), 
c\in D_*(f),d\in D(f)$.
Of course, $a,b,c$, and $d$ are elements of $W_G$.
By (\ref{A:BCR}) we know that 
\begin{align*}
x \leq y\ \text{ in }\ \rs \iff e\leq f,\ a \leq cw,\ w^{-1} d^{-1} \leq b^{-1},
\end{align*}
for some $w\in W_G(f)W_G(e)$. 
The idempotents of $\rs$ are in $\mathcal{R}_n$, hence, the Bruhat-Chevalley-Renner order on them is the one that is induced from $\mathcal{R}_n$.  
It follows from (\ref{A:inCn}) that the relations $a \leq cw$ and $w^{-1} d^{-1} \leq b^{-1}$ hold in $W_G$ if and only if they hold in 
$\mathcal{S}_n \subseteq \mathcal{R}_n$. 
Therefore, the relation $x\leq y$ holds in $\rs$ if and only if it holds in $\mathcal{R}_n$. 
\end{proof}

\begin{Corollary}
Let $x=(x_1, \dots , x_n)$ and $y=(y_1, \dots , y_n)$ be two elements from $\rs$. 
Then $x \leq_{BCR} y$ if and only if for every $i\in \{1,\dots, n-1\}$ we have $\tilde{x}(i) \leqslant \tilde{y}(i)$. 
\end{Corollary}
\begin{proof}
It follows from Theorem~\ref{T:inrsn} that $x\leq_{BCR} y$ in $\rs$ if and only if $x\leq_{BCR} y$ in $\mathcal{R}_n$. 
The rest of the proof follows from Theorem~\ref{T:CR}.
\end{proof}

\section{The Borel Submonoid of $MSp_n$}\label{S:MSp_n}

We will follow our convention that if the even integer $n=2l$ is fixed, then $G$ stands for $Sp_n$. 
We will denote by $B_G$, as before, the Borel subgroup in $Sp_n$ that is defined by 
\[
B_G:=B_n\cap Sp_n. 
\]
The Borel subgroup of the unit group $GSp_n$ of $MSp_n$ is given by $B:= \C^* B_G$, 
and the corresponding Borel submonoid of $MSp_n$ is the Zariski closure of $B$ in $MSp_n$. 
Evidently, $B$ is a connected, hence irreducible, algebraic group. 
Then $\overline{B}$ is an irreducible $B\times B$-variety. 
The orbits of $B\times B$ are parametrized by $x\in \rs$ such that $x\leq 1$. 
Indeed, we have 
\[
\mc{B}_G = \mc{B}_n \cap \mc{R}_G := \{ x \in \rs :\ x\leq 1\}.
\]
We depict the Bruhat-Chevalley-Renner order on $\mc{B}_{Sp_4}$ in Figure~\ref{F:BSp4}.

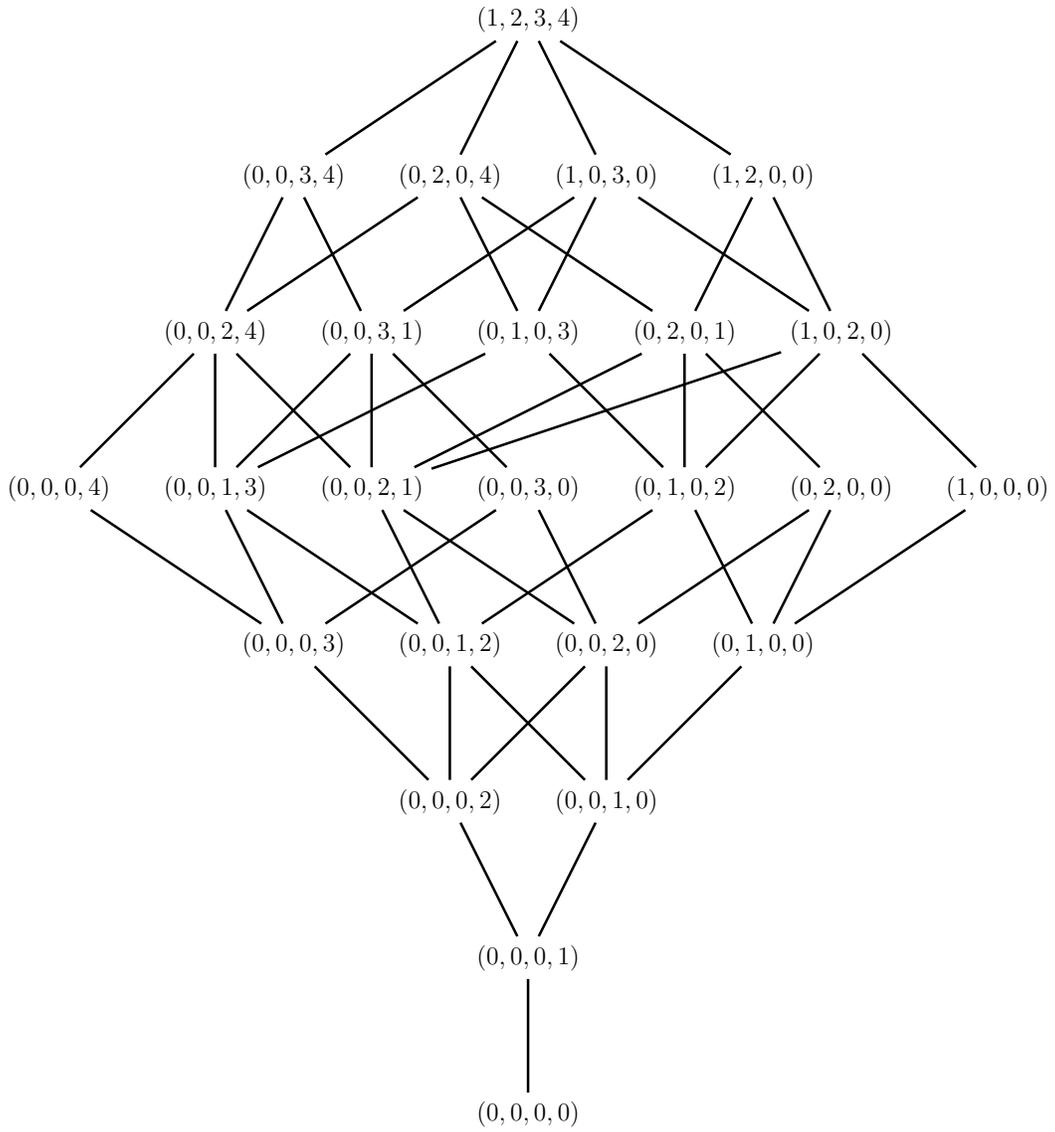
\begin{figure}[htp]
\begin{center}

\scalebox{.8}{
\begin{tikzpicture}[scale=.65]

\node at (0,28) (h1) {$(1,2,3,4)$};

\node at (-6,24) (g1) {$(0,0,3,4)$};
\node at (-2,24) (g2) {$(0,2,0,4)$};
\node at (2,24) (g3) {$(1,0,3,0)$};
\node at (6,24) (g4) {$(1,2,0,0)$};

\node at (-8,20) (f1) {$(0,0,2,4)$};
\node at (-4,20) (f2) {$(0,0,3,1)$};
\node at (0,20) (f3) {$(0,1,0,3)$};
\node at (4,20) (f4) {$(0,2,0,1)$};
\node at (8,20) (f5) {$(1,0,2,0)$};

\node at (-12,16) (e1) {$(0,0,0,4)$};
\node at (-8,16) (e2) {$(0,0,1,3)$};
\node at (-4,16) (e3) {$(0,0,2,1)$};
\node at (0,16) (e4) {$(0,0,3,0)$};
\node at (4,16) (e5) {$(0,1,0,2)$};
\node at (8,16) (e6) {$(0,2,0,0)$};
\node at (12,16) (e7) {$(1,0,0,0)$};

\node at (-6,12) (d1) {$(0,0,0,3)$};
\node at (-2,12) (d2) {$(0,0,1,2)$};
\node at (2,12) (d3) {$(0,0,2,0)$};
\node at (6,12) (d4) {$(0,1,0,0)$};

\node at (-2,8) (c1) {$(0,0,0,2)$};
\node at (2,8) (c2) {$(0,0,1,0)$};

\node at (0,4) (b1) {$(0,0,0,1)$};

\node at (0,0) (a) {$(0,0,0,0)$};

\draw[-, very thick] (a) to (b1);
\draw[-, very thick] (b1) to (c1);
\draw[-, very thick] (b1) to (c2);
\draw[-, very thick] (c1) to (d1);
\draw[-, very thick] (c1) to (d2);
\draw[-, very thick] (c1) to (d3);

\draw[-, very thick] (c2) to (d2);
\draw[-, very thick] (c2) to (d3);
\draw[-, very thick] (c2) to (d4);

\draw[-, very thick] (d1) to (e1);
\draw[-, very thick] (d1) to (e2);
\draw[-, very thick] (d1) to (e4);
\draw[-, very thick] (d2) to (e2);
\draw[-, very thick] (d2) to (e3);
\draw[-, very thick] (d2) to (e5);

\draw[-, very thick] (d3) to (e3);
\draw[-, very thick] (d3) to (e4);
\draw[-, very thick] (d3) to (e6);

\draw[-, very thick] (d4) to (e5);
\draw[-, very thick] (d4) to (e6);
\draw[-, very thick] (d4) to (e7);

\draw[-, very thick] (e1) to (f1);
\draw[-, very thick] (e2) to (f1);
\draw[-, very thick] (e2) to (f2);
\draw[-, very thick] (e2) to (f3);
\draw[-, very thick] (e3) to (f1);
\draw[-, very thick] (e3) to (f2);
\draw[-, very thick] (e3) to (f4);
\draw[-, very thick] (e3) to (f5);
\draw[-, very thick] (e4) to (f2);
\draw[-, very thick] (e5) to (f3);
\draw[-, very thick] (e5) to (f4);
\draw[-, very thick] (e5) to (f5);
\draw[-, very thick] (e6) to (f4);
\draw[-, very thick] (e7) to (f5);

\draw[-, very thick] (f1) to (g1);
\draw[-, very thick] (f1) to (g2);
\draw[-, very thick] (f2) to (g1);
\draw[-, very thick] (f2) to (g3);
\draw[-, very thick] (f3) to (g2);
\draw[-, very thick] (f3) to (g3);
\draw[-, very thick] (f4) to (g2);
\draw[-, very thick] (f4) to (g4);
\draw[-, very thick] (f5) to (g3);
\draw[-, very thick] (f5) to (g4);

\draw[-, very thick] (g1) to (h1);
\draw[-, very thick] (g2) to (h1);
\draw[-, very thick] (g3) to (h1);
\draw[-, very thick] (g4) to (h1);

\end{tikzpicture}
}
\end{center}
\caption{Bruhat-Chevalley-Renner order on $\mc{B}_{Sp_4}$.}
\label{F:BSp4}
\end{figure}

We are going to reformulate the description of $\mc{B}_G$ in two different ways. 
\begin{enumerate}
\item  
It is observed in~\cite[Lemma 2.3]{PPR97} that an element $r$ from $\rs$ satisfies $r\leq 1$ if and only if 
$a\leq b$, where $a e b^{-1}$ is the standard form of $r$. 
Thus, we have 
\begin{align}\label{A:1}
\bs = \{ a e b^{-1} :\ aeb^{-1}\text{ is in standard form, $e\in \Lambda_G$, $a\in D_*(e), b\in D(e)$, and $a\leq b$}\}.
\end{align}

\item Let $D(x)$ (resp. $R(x)$) denote the domain (resp. range) 
of an element $x\in \mathcal{R}_n$.
The data of $D(x)$ and $R(x)$ are not enough to recover $x$.
One needs to know the (bijective) assignment between them to uniquely determine $x$. 
Let $D(x)$ and $R(x)$ be as follows: 
\[
D(x) = \{i_1,\dots, i_k\}\ \text{ and } R(x) = \{j_1,\dots, j_k\},
\]
where 
\[
x(i_t) = j_t \ \text{ for } t\in \{1,\dots, k\}.
\]
We will assume that the entries of $D(x)$ are listed in the increasing order 
as in $1\leq i_1 < \cdots < i_k \leq n$,
however, we may not have the same ordering on the corresponding elements of $R(x)$.

Notice that in order for $x$ be $\leq 1$ in the Bruhat-Chevalley-Renner order its matrix 
representation has to have all of its nonzero entries on or above the main diagonal. 
Since $D(x)$ gives the column indices of the nonzero entries in $x$, and since 
$R(x)$ gives the row indices of the nonzero entries in $x$, we see that 
\[
x\leq 1 \iff i_t \geq j_t\ \text{ for every $t\in \{1,\dots, k\}$.}
\]
By (\ref{A:LiRenner}), for the elements $x$ in $\rs\setminus \{1\}$, 
both of the subsets $D(x),R(x) \subseteq \{1,\dots, n\}$ are admissible. 

\end{enumerate}

\begin{Question}\label{Q:leq1}
What is the cardinality of $\bs$? 
By the second item, our problem is equivalent to the question of 
finding, for every $k\in \{1,\dots, l\}$, the number of pairs of admissible subsets 
\[
I = \{i_1,\dots, i_k\}\ \text{ and } J = \{j_1,\dots, j_k\} 
\]
such that 
\begin{enumerate}
\item $1\leq i_1 < \cdots < i_k \leq n$ and $J\subset \{1,\dots, n\}$; there are no order constraints on the elements of $J$.
\item $i_t \geq j_t \geq 1$ for every $t\in \{1,\dots, k\}$.
\end{enumerate}
\end{Question}

It turns out that the number of admissible subsets of $\{1,\dots,n\}$ has a pleasant formula. 
We begin with a simple lemma.

\begin{Lemma}\label{L:a_nk's}
Let $n$ and $k$ be two positive integers such that $1\leq k \leq n$. 
Assume that $n$ is an even number, $n=2l$. 
Let $A_{n,k}$ denote the set of admissible subsets of $\{1,\dots, n\}$ with $k$ 
elements. 
If $a_{n,k}$ denotes the cardinality of $A_{n,k}$, then 
\[
a_{n,k}= \sum_{r=0}^k {l \choose r }{l-r \choose k-r} = {l \choose k} 2^k.
\]
\end{Lemma}

\begin{proof}
Clearly, by the pigeon-hole principle, if $k > l$, then $A_{n,k} = \emptyset$. 
Also, in this case, ${l-r \choose k-r}$ is 0 for every $r\in \{0,\dots, k\}$, hence, $a_{n,k} = 0$. 
Therefore, we will assume that $k\leq l$.

Let $A= \{i_1,\dots, i_k\}$ ($i_1 < \cdots < i_k$) be an element from $A_{n,k}$. 
The entries of $A$ satisfy the inequalities 
\[
1\leq i_1<\cdots < i_r \leq l < i_{r+1} <\cdots <i_k \leq n.
\]
We will determine the number of such $A$. 
Clearly, the first $r$ entries, $i_1,\dots, i_r$, can be chosen in ${l\choose r}$ ways. 
Then the remaining entries, $i_{r+1},\dots, i_k$, cannot be contained in the set $\{ \theta(i_s) = n-i_s +1: s\in \{1,\dots, r\}\}$. 
In other words, $\{ i_{r+1},\dots, i_k \} \subseteq \{ l+1,\dots, 2l\}\setminus \{ \theta(i_s) = n-i_s +1: s\in \{1,\dots, r\}\}$. 
Then, the number of possibilities for $\{ i_{r+1},\dots, i_k\}$ is given by 
${l-r \choose k-r}$. 
Therefore, in total, we have $\sum_{r=0}^k {l \choose r}{l-r \choose k-r}$ possibilities for $A$. This finishes the proof of the first equality. 
To prove the second equality, we manipulate the summation as follows:
\begin{align}\label{A:beforeafter}
 \sum_{r=0}^k {l \choose r }{l-r \choose k-r} = 
  \sum_{r=0}^k \frac{l!}{(l-r)! r!} \frac{(l-r)!}{(l-k)! (k-r)!}=  \sum_{r=0}^k \frac{l!}{r! (l-k)! (k-r)!}.
\end{align}
Let us multiply and divide each summand in the last sum in (\ref{A:beforeafter}) by $k!$. Then by reorganizing the terms we get  
\begin{align*}
\sum_{r=0}^k \frac{l!}{(l-k)! k!} \frac{k!}{r! (k-r)!}
=\sum_{r=0}^k {l \choose k} {k\choose r} 
={l \choose k}\sum_{r=0}^k {k\choose r} 
= {l \choose k} 2^k.
\end{align*}
This finishes the proof.
\end{proof}

Since the empty set is admissible, we set $a_{n,0} = 1$. 
\begin{Corollary}\label{C:nofadmissible}
The total number of admissible subsets of $\{1,\dots, n\}$, that is, $|\cup_{k=0}^l A_{n,k}|$, is 
equal to $3^l$. 
\end{Corollary}
\begin{proof}
We will determine the number 
$\sum_{k=0}^l a_{n,k} = \sum_{k=0}^l  {l \choose k} 2^k$.
But, by the binomial theorem, $f(2)=\sum_{k=0}^l  {l \choose k} 2^k$, 
where $f(x) = (1+x)^l$.
\end{proof}

\begin{Theorem}\label{T:maxelements}
The $k$-th symplectic Stirling poset $\mc{B}_G(k)$ is a graded bounded poset with a unique minimum element.
There are $ {l \choose k} 2^k$ maximal elements in $\mc{B}_G(k)$. 
\end{Theorem}

\begin{proof}

If $k=n$ (resp. $k=0$), then $\mc{B}_G(k) = \{ id \}$ (resp. $\mc{B}_G(k) = \{ \mathbf{0} \}$), hence, in these cases 
there is nothing to prove. 
We proceed with the assumption that $1\leq k \leq l$. 
Since $\mc{B}_G$ is equal to the intersection $\mc{R}_G\cap \mc{B}_n$, we have
\[
\mc{B}_G(k) = \mc{B}_n(k) \cap \mc{R}_G\ \text{ for $k\in\{1,\dots, l\}$}.
\]
Notice that the rook $id(k):=(0,\dots, 0, 1,2,\dots, k)$ is a symplectic rook.
In fact, $id(k)$ is the unique minimum of $\mc{B}_n(k)$. 
It follows from Theorem~\ref{T:inrsn} that $id(k)$ is the unique minimum element in $\mc{B}_G(k)$ as well. 
Next, we will show that $\mc{B}_G(k)$ is a graded poset. 
To this end, it will suffice to show that every maximal element of $\mc{B}_G(k)$ has the same rank. 
In~\cite[Lemma 5.1]{CanCherniavsky}, it is shown that the maximal elements of $\mc{B}_G(k)$ are given by 
the diagonal idempotents of rank $k$ in $\mc{R}_n$. 
Clearly, any diagonal idempotent of rank $k$ whose domain and range are admissible subsets in $\{1,\dots, n\}$ 
are contained in $\mc{B}_G(k)$. 
But for a diagonal matrix, the domain and the range agree, therefore, the number of diagonal idempotents in $\mc{B}_G(k)$ 
is equal to the number of admissible subsets in $\{1,\dots, n\}$. 
This number is equal to ${l\choose k}2^k$ by Lemma~\ref{L:a_nk's}. 
Next, we will show that there are no other maximal elements in $\mc{B}_G(k)$. 
Towards a contradiction, let $a= (a_1,\dots, a_n)$ be a maximal element in $\mc{B}_G(k)$ which is not a diagonal idempotent. 
Since $a$ is an upper triangular rook, we know that $a_i \leq i$ for every $i\in \{1,\dots, n\}$. 
Let $i\in \{1,\dots, n\}$ be the smallest index such that $0<a_i \lneq i$. Let $j$ denote $a_i$. 
Then we know that $a_j=0$.
Let $b$ denote the rook matrix that is obtained by interchanging $a_i$ and $a_j$. 
It is easy to verify that $b$ is an element of $\mc{B}_G(k)$ such that $a < b$. 
This contradicts with our assumption that $a$ is a maximal element. 
Hence, the proof of our theorem is finished.

\end{proof}

\begin{Remark}
For each $d\in \{1,\dots, l\}$, by (\ref{A:partialorder}), we have a very special poset structure 
on $A_{n,d}$. 
It is easily seen from~\cite[Section 6.1.1]{LakshmibaiRaghavan} that $(A_{n,d},\leqslant)$ 
is isomorphic to the Bruhat-Chevalley order on the Grassmann variety $G/P_d$,
where $G=Sp_n$ and $P_d$ is the maximal parabolic subgroup corresponding to the 
set of simple generators $S_G \setminus \{s_d\}$ in $W_G$.
\end{Remark}

\begin{Proposition}\label{P:A}
Let $\Lambda_G$ be the cross section lattice of $G$ as in (\ref{A:crossG}).
If $e$ is an element from $\Lambda_G\setminus \{1\}$, then $W_G(e)$ is a maximal parabolic subgroup in $W_G$. 
Conversely, any maximal parabolic subgroup of $W_G$ is obtained this way. 
\end{Proposition}
\begin{proof}
The cross section lattice $\Lambda_G$ is part of the cross section lattice of the monoid $M_n$.
It is easy to verify that, if the matrix rank of $e$ is $d$, then the centralizer of $e$ in $W_H = S_n$ is 
the maximal parabolic subgroup generated by the set $\{r_1,\dots,r_{n-1}\} \setminus \{r_d\}$.  
Let $s_1,\dots, s_l$ denote, as defined in (\ref{A:s}), the simple Coxeter generators for $W_G$.
Now it is easy to check that, for every $j \in \{1,\dots, l\} \setminus \{d\}$, we have 
\[
s_j e = e s_j,\ \text{ hence, } \ W_G(e) = \langle s_j :\ j\in \{1,\dots, l\} \setminus \{ d\} \rangle.
\]
Our second assertion is now easy to verify. This finishes the proof.
\end{proof}

Next, we will compute the stabilizer of an element $e$ from $\Lambda_G\setminus \{1\}$. 

\begin{Proposition}\label{P:B}
Let $e$ be an element from $\Lambda_G \setminus \{1\}$. 
If the matrix rank of $e$ is $d$, where $1\leq d < l$, then $(W_G)_*(e)$ is generated by 
the simple Coxeter generators $s_{d+1},\dots, s_l$.
If $d=l$, then $(W_G)_*(e) = \{1\}$.  
\end{Proposition}

\begin{proof}
Once again, the proof will follow from the corresponding computation that is performed in the rook monoid
(the Renner monoid of $M_n$). 
In that case, by explicitly computing the matrix products $r_j e$ ($j\in \{1,\dots, n-1\}$),
one sees that $r_j e = e r_j = e$ if and only if $j\in \{d+1,\dots, n-1\}$.
It follows immediately from this observation that $s_j = r_j r_{n-j}$ stabilizes $e$ if and only if $j \in \{d+1,\dots, l\}$.
It also follows that if the rank of $e$ is $l$, then $s_l$ does not stabilize $e$, hence, $(W_G)_*(e) = \{1\}$. 
This finishes the proof. 
\end{proof}

\section{Folding, unfolding}

Let $n$ be a positive integer. 
A collection $S_1,\dots, S_r$ of subsets of the set $S:=\{1,\dots, n\}$ is said to 
be a {\em set partition} of $S$ if $S_i$'s ($i=1,\dots, r$) are mutually disjoint and $\cup_{i=1}^r S_i = S$. 
In this case, the $S_i$'s are called the blocks of the partition.
The collection of all set partitions of $S$ is denoted by $\Pi_n$. 
We will often drop set parentheses and commas and just put vertical bars between blocks. 
If $S_1,\dots, S_k$ are the blocks of a set partition $\pi$ from $\Pi_n$, 
then the {\em standard form} of $\pi$ is defined as $S_1|S_2|\cdots |S_k$, where 
we assume that $\min S_1 <\cdots < \min S_k$
and the elements of each block are listed in increasing order. 
For example, $\pi = 136 | 2459 | 78$ is a set partition from $\Pi_9$. 
Set partitions can be visualized by using ``arc-diagrams'' which we will define next.

A linearly ordered poset is called a {\em chain}. 
We will identify chains by their Hasse diagrams; 
we draw a Hasse diagram by placing the smallest 
entry on the left and connecting the vertices by arcs.
For example, in Figure~\ref{F:chain1}, we have the chain on 9 vertices,
where each arc represents a covering relation.
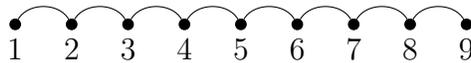
\begin{figure}[h]
\begin{center}
\begin{tikzpicture}[scale=.75]
\node[below=.05cm] at (0,0) {$1$};
\node[draw,circle, inner sep=0pt, minimum width=4pt, fill=black] (0) at (0,0) {};
\node[below=.05cm] at (1,0) {$2$};
\node[draw,circle, inner sep=0pt, minimum width=4pt, fill=black] (1) at (1,0) {};
\node[below=.05cm] at (2,0) {$3$};
\node[draw,circle, inner sep=0pt, minimum width=4pt, fill=black] (2) at (2,0) {};
\node[below=.05cm] at (3,0) {$4$};
\node[draw,circle, inner sep=0pt, minimum width=4pt, fill=black] (3) at (3,0) {};
\node[below=.05cm] at (4,0) {$5$};
\node[draw,circle, inner sep=0pt, minimum width=4pt, fill=black] (4) at (4,0) {};
\node[below=.05cm] at (5,0) {$6$};
\node[draw,circle, inner sep=0pt, minimum width=4pt, fill=black] (5) at (5,0) {};
\node[below=.05cm] at (6,0) {$7$};
\node[draw,circle, inner sep=0pt, minimum width=4pt, fill=black] (6) at (6,0) {};
\node[below=.05cm] at (7,0) {$8$};
\node[draw,circle, inner sep=0pt, minimum width=4pt, fill=black] (7) at (7,0) {};
\node[below=.05cm] at (8,0) {$9$};
\node[draw,circle, inner sep=0pt, minimum width=4pt, fill=black] (8) at (8,0) {};
\draw[color=black] (0) to [out=60,in=120] (1);
\draw[color=black] (1) to [out=60,in=120] (2);
\draw[color=black] (2) to [out=60,in=120] (3);
\draw[color=black] (3) to [out=60,in=120] (4);
\draw[color=black] (4) to [out=60,in=120] (5);
\draw[color=black] (5) to [out=60,in=120] (6);
\draw[color=black] (6) to [out=60,in=120] (7);
\draw[color=black] (7) to [out=60,in=120] (8);
\end{tikzpicture}
\caption{A chain on 9 vertices.}
\label{F:chain1}
\end{center}
\end{figure}

\begin{Definition}
A {\em labeled chain} is a chain whose vertices are labeled by distinct numbers. 
An {\em arc-diagram on $n$ vertices} is a disjoint union of labeled chains where the labels are from $\{1,\dots,n\}$
and each label $i\in \{1,\dots, n\}$ is used exactly once. 
We depict an example in Figure~\ref{F:introexample}.  
\end{Definition}
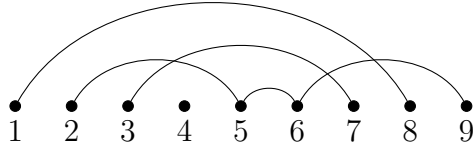
\begin{figure}[h]
\begin{center}
\begin{tikzpicture}[scale=.75]
\node[below=.05cm] at (0,0) {$1$};
\node[draw,circle, inner sep=0pt, minimum width=4pt, fill=black] (0) at (0,0) {};
\node[below=.05cm] at (1,0) {$2$};
\node[draw,circle, inner sep=0pt, minimum width=4pt, fill=black] (1) at (1,0) {};
\node[below=.05cm] at (2,0) {$3$};
\node[draw,circle, inner sep=0pt, minimum width=4pt, fill=black] (2) at (2,0) {};
\node[below=.05cm] at (3,0) {$4$};
\node[draw,circle, inner sep=0pt, minimum width=4pt, fill=black] (3) at (3,0) {};
\node[below=.05cm] at (4,0) {$5$};
\node[draw,circle, inner sep=0pt, minimum width=4pt, fill=black] (4) at (4,0) {};
\node[below=.05cm] at (5,0) {$6$};
\node[draw,circle, inner sep=0pt, minimum width=4pt, fill=black] (5) at (5,0) {};
\node[below=.05cm] at (6,0) {$7$};
\node[draw,circle, inner sep=0pt, minimum width=4pt, fill=black] (6) at (6,0) {};
\node[below=.05cm] at (7,0) {$8$};
\node[draw,circle, inner sep=0pt, minimum width=4pt, fill=black] (7) at (7,0) {};
\node[below=.05cm] at (8,0) {$9$};
\node[draw,circle, inner sep=0pt, minimum width=4pt, fill=black] (8) at (8,0) {};
\draw[color=black] (6) to [out=120,in=60] (2);
\draw[color=black] (7) to [out=120,in=60] (0);
\draw[color=black] (4) to [out=120,in=60] (1);
\draw[color=black] (5) to [out=120,in=60] (4);
\draw[color=black] (8) to [out=120,in=60] (5);
\end{tikzpicture}
\caption{An arc-diagram on 9 vertices}
\label{F:introexample}
\end{center}
\end{figure}
Clearly, in an arc-diagram subchains represents the blocks of the corresponding set partition.
We know from~\cite[Lemma 1.17]{Bona} that the number of set partitions of $S=\{1,\dots, n\}$ into $k$ blocks, denoted by $S(n,k)$, and called the {\em $(n,k)$-th Stirling number of the second kind}, is given by the formula $S(n,k) = \frac{1}{k!} \sum_{i=1}^k (-1)^i {k\choose i} (k-i)^n$.
The recurrence formula for the Stirling numbers of the second kind is well-known:
\begin{align*}
S(l+1,k) = S(l,k-1) + k S(l,k),
\end{align*}
where 
\begin{align*}
S(l,k) = 
\begin{cases}
1 & \text{ if $l=k=0$};\\
0 & \text{ if $l>0$ and $k=0$};\\
0 & \text{ if $l<0$ or $k<0$ or $l<k$}.
\end{cases}
\end{align*}

Let $\mc{B}_{n}$ denote the submonoid of $\mc{R}_n$ such that if $x\in \mc{B}_n$, then $x$ is an upper triangular matrix.
The subsemigroup of nilpotent elements in $\mc{B}_n$ will be denoted by $\mc{B}_n^{nil}$.
For each $A$ in $\mc{B}_{n}$, there exists a unique $(n+1)\times (n+1)$ nilpotent matrix, $\tilde{A}$, 
which is obtained from $A$ by appending to it a column and a row of zeros as follows: 
\begin{align}\label{A:AtoAtilde}
A \longmapsto \tilde{A} :=
\begin{bmatrix}
0 & & \\
\vdots & A & \\
0 & \dots & 0
\end{bmatrix} \in \mc{B}_{n+1}\qquad (A\in \mc{B}_{n}).
\end{align}
In this notation, it is easily verified that (\ref{A:AtoAtilde}) defines a set bijection 
$\mc{B}_{n} \longrightarrow \mc{B}_{n+1}^{nil}$. 
There is a simple bijection between $\mc{B}^{nil}_{n+1}$ and $\Pi_{n+1}$ which is defined as follows:
the matrix corresponding to the set partition $A$ has an entry equal to $1$ in row $i$ and and column $j$ if and only if
$(i,j)$ is an arc of $A$. Therefore, for $k \in \{1,\dots, n+1\}$, we have 
\begin{align}\label{A:Stirlingcount}
S(n+1,k) = | \{ A\in \mc{B}_n :\ \text{rank} A = n+1-k \} |.
\end{align}
It follows from the bijections above that the number of elements of $\mc{B}_n$ is given by the summation $b_{n+1}:=\sum_{k=0}^{n+1} S(n+1, k)$, 
which is called the {\em $(n+1)$th Bell number}. 
As a convention, we set $b_0=1$ and $b_1 =1$.

It is easy to check that the number of elements of $\mc{R}_n$ of rank $k$ is given by 
\begin{align}\label{A:rankkinRn}
| \{ A\in \mc{R}_n :\ \text{rank}(A)= k\} | = {n \choose k} \frac{n!}{(n-k)!}.
\end{align}
We will express this cardinality by using Stirling numbers of the second kind.

Every element $A$ of $\mc{R}_n$ has a triangular decomposition in $\mc{R}_n$,
\begin{align}\label{A:triangular}
A= A_l + A_d + A_u,
\end{align}
where $A_l$ is a strictly lower triangular matrix, $A_d$ is a diagonal matrix, and $A_u$ is a strictly upper triangular matrix. 

\begin{Proposition}\label{P:triangular}
Let $S_{a,b,c}(n)$ denote the number of elements $A\in \mc{R}_n$ such that 
$\text{rank} (A_l) = a, \text{rank} (A_d) = b$, and $\text{rank} (A_u) = c$, 
where $A_l,A_d$, and $A_u$ are as in (\ref{A:triangular}). 
Then we have 
\begin{align*}
{n \choose k} \frac{n!}{(n-k)!} & = \sum_{a+b+c = k} S_{a,b,c}(n) \\
&=  \sum_{a+b+c = k} {n \choose b} S(n+1,n+1-a) S(n+1,n+1-c).
\end{align*}
\end{Proposition}

\begin{proof}
The number of strictly upper triangular elements of rank $k$ in $\mc{R}_n$ is equal to 
the number of strictly lower triangular rank $k$ elements in $\mc{R}_n$. 
Now the proof of the first equality follows from the equality in (\ref{A:rankkinRn}) and the uniqueness of the triangular decomposition in (\ref{A:triangular}).
The proof of the second equality follows from (\ref{A:Stirlingcount}) together with the fact that there are exactly 
${n \choose b}$ diagonal 0/1 matrices of rank $b$.
\end{proof}

We proceed with the assumption that $n$ is an even number of the form $n=2l$, $l\in \Z_+$. 
In the sequel, we will count the number of elements of $\mc{B}_G$, where $G=Sp_n$, 
by a technique that we call {\em unfolding}.
But before that we want to demonstrate that the elements of $\mc{R}_G$ behave well under ``folding''.
We already mentioned the result of Li and Renner~\cite[Theorem 3.1.10]{LiRenner}, which states that, 
if an element $A$ from $\mc{R}_G$ is singular, then both of the domain and the range of $A$ are admissible
subsets in $\{1,\dots, n\}$. 
Furthermore, the elements of $D(A)$ correspond to the indices of the nonzero columns of $A$, 
and the elements of $R(A)$ correspond to the indices of the nonzero rows of $A$.  
This shows that $A$ can be folded vertically as well as horizontally. 
We will demonstrate what we mean here by an example. 
\begin{Example}
In this example, we fold an element of $\mc{R}_{Sp_8}$ horizontally from top to bottom: 
\begin{center}
\begin{tikzpicture}[node distance=-1ex]
  \matrix (mymatrix) at (-5,0) [matrix of math nodes,left delimiter={[},right
delimiter={]}]
  {
1 & 0 & 0 & 0 & 0 & 0 & 0 & 0 \\
0 & 0 & 0 & 0 & 1 & 0 & 0 & 0 \\
0 & 0 & 0 & 0 & 0 & 0  & 0 & 0 \\
0 & 0 & 0 & 0 & 0 & 0  & 0 & 0 \\
0 & 0 & 1 & 0 & 0 & 0  & 0 & 0 \\
0 & 0 & 0 & 0 & 0 & 0  & 1 & 0 \\
0 & 0 & 0 & 0 & 0 & 0  & 0 & 0 \\
0 & 0 & 0 & 0 & 0 & 0  & 0 & 0 \\
  };
  
\draw[thick, dashed, red] (mymatrix-4-1.south west) -- (mymatrix-4-8.south east);
\draw[thick, ->] (-2,0) -- (2,0);
  \matrix (mysecondmatrix) at (5,0) [matrix of math nodes,left delimiter={[},right
delimiter={]}]
  {
0 & 0 & 1 & 0 & 0 & 0 & 0 & 0 \\
0 & 0 & 0 & 0 & 0 & 0 & 1 & 0 \\
0 & 0 & 0 & 0 & 1 & 0  & 0 & 0 \\
1 & 0 & 0 & 0 & 0 & 0  & 0 & 0 \\
  };
\end{tikzpicture}
\end{center}
\end{Example}

\begin{Example}
In this example, we fold the matrix of the previous example vertically from left to right: 
\begin{center}
\begin{tikzpicture}[node distance=-1ex]
  \matrix (mymatrix) at (-5,0) [matrix of math nodes,left delimiter={[},right
delimiter={]}]
  {
1 & 0 & 0 & 0 & 0 & 0 & 0 & 0 \\
0 & 0 & 0 & 0 & 1 & 0 & 0 & 0 \\
0 & 0 & 0 & 0 & 0 & 0  & 0 & 0 \\
0 & 0 & 0 & 0 & 0 & 0  & 0 & 0 \\
0 & 0 & 1 & 0 & 0 & 0  & 0 & 0 \\
0 & 0 & 0 & 0 & 0 & 0  & 1 & 0 \\
0 & 0 & 0 & 0 & 0 & 0  & 0 & 0 \\
0 & 0 & 0 & 0 & 0 & 0  & 0 & 0 \\
  };
  
\draw[thick, dashed, red] (mymatrix-1-4.north east) -- (mymatrix-8-4.south east);
\draw[thick, ->] (-2,0) -- (1,0);
  \matrix (mysecondmatrix) at (3,0) [matrix of math nodes,left delimiter={[},right
delimiter={]}]
  {
0 & 0 & 0 & 1 \\
1 & 0 & 0 & 0 \\
0 & 0  & 0 & 0 \\
0 & 0  & 0 & 0 \\
0 & 1  & 0 & 0 \\
0 & 0  & 1 & 0 \\
0 & 0  & 0 & 0 \\
0 & 0  & 0 & 0 \\
  };
\end{tikzpicture}
\end{center}
\end{Example}

\begin{Definition}
We will denote the horizontal folding operation from top to bottom by $F_{TB}$.
Likewise, we will denote the vertical folding operation from left to right by $F_{LR}$. 
\end{Definition}

Clearly, the operations $F_{TB}$ and $F_{LR}$ can be composed.
In fact, they commute, 
\begin{align}\label{A:composition}
F_{TB} F_{LR} = F_{LR} F_{TB}.
\end{align}

Let $F$ denote the composition of the folding operators as in (\ref{A:composition}).
We will refer to $F$ by the {\em folding map}.

\begin{Proposition}
The folding map is a surjective map from $\mc{R}_{Sp_n}$ onto the rook monoid $\mc{R}_l$. 
Furthermore, the restricted folding map, $F\vert_{\mc{B}_G}$, which we will denote by  $F'$, 
is surjective as well. 
\end{Proposition}

\begin{proof}
If we show that $F'$ is surjective, then the surjectivity of $F$ will follow.
To this end, let $A$ be an element from $\mc{R}_l$, and let $A= A_l + A_d + A_u$ be its triangular decomposition. 
Recall from the introduction that $J_l$ denotes the $l\times l$ permutation matrix with 1's on its anti-diagonal.
Now we define an $n\times n$ matrix $B$ as follows:
$
B:= 
\begin{bmatrix}
\mathbf{0} & \widetilde{A_l} \\
\mathbf{0} & A_d + A_u 
\end{bmatrix},
$
where $\mathbf{0}$ is the $l\times l$ 0 matrix, and $\widetilde{A_l} := J_l A_l$.
In other words, $\widetilde{A_l}$ is the matrix whose $i$th row is the $(l-i+1)$th row of $A_l$.
Since the indices of the nonzero columns of $B$ are contained in the set $\{l+1,\dots, n\}$,
the domain of $B$ is an admissible subset in $\{1,\dots, n\}$. 
It is also easy to check that the set of row indices of $B$ is an admissible subset of $\{1,\dots, n\}$.
Clearly, $B$ is upper triangular matrix, therefore, $B\in \mc{B}_G$. 
Finally, by its construction, the image of $B$ under $F'$ is equal to $A$, $F'(B) =F(B)= A$. 
This finishes the proofs of our assertions. 
\end{proof}

We are now ready to count the number of elements of $\mc{B}_G$ by ``unfolding'' the elements of 
$\mc{R}_l$ first horizontally from bottom to top, and then vertically from right to left. 
We will demonstrate our count by an example.

\begin{Example}\label{E:indicates}
We will compute the preimage of $J_2 =\begin{bmatrix} 0 & 1 \\ 1 & 0 \end{bmatrix}$ 
under the restricted folding map $F' : \mc{B}_{Sp_4}\to \mc{R}_2$.
Equivalently, we will determine the set $F_{LR}^{-1} (F_{TB}^{-1}(J_2)) \cap \mc{B}_{Sp_4}$. 
Since we are looking for the upper triangular elements in the preimage, the lower halves 
of the $4\times 2$ matrices in $F_{TB}^{-1}(J_2)$ must be upper triangular. 
The following matrices are the possibilities: 

\begin{center}
\begin{tikzpicture}[node distance=-1ex]
\matrix (mymatrix1) at (-5,0) [matrix of math nodes,left delimiter={[},right
delimiter={]}]
  {
1 & 0 \\
\node (a21) {0}; & \node (a22) {0}; \\  
0 & 1 \\
0 & 0 \\
  };
  \draw [dashed,red,thick] (a21.south west) -- (a22.south east);
\matrix (mymatrix2) at (5,0) [matrix of math nodes, left delimiter={[},right
delimiter={]}]
  {
1 & 0 \\
\node (b21) {0}; & \node (b22) {1}; \\  
0 & 0 \\
0 & 0 \\
  };
    \draw [dashed,red,thick] (b21.south west) -- (b22.south east);
\matrix (mymatrix) at (0,-4) [matrix of math nodes,left delimiter={[},right
delimiter={]}]
  {
0 & 1 \\
1 & 0 \\
  };
\draw[thick, ->] (-3.5,-1.5) -- (-1,-3);
\draw[thick, ->] (3.5,-1.5) -- (1,-3);
\node at (-2,-2) {$F_{TB}$};
\node at (1.8,-2) {$F_{TB}$};
\end{tikzpicture}
\end{center}
Let $A_1$ denote the $4\times 2$ matrix that is on the top-left position, 
and let $A_2$ denote the $4\times 2$ matrix that is on the top-right position.
Then, the following two matrices are mapped onto $A_1$ by $F_{LR}$:
\begin{center}
\begin{tikzpicture}[node distance=-1ex]
\matrix (mymatrix1) at (-4.5,0) [matrix of math nodes,left delimiter={[},right
delimiter={]}]
  {
0 & \node (a12) {0}; & 1 & 0 \\
0 & 0 & 0 & 0 \\
0 & 0 & 0 & 1 \\
0 & \node (a42) {0}; & 0 & 0 \\
  };
  \draw [dashed,red,thick] (a12.north east) -- (a42.south east);
\matrix (mymatrix2) at (-1,0) [matrix of math nodes, left delimiter={[},right
delimiter={]}]
  {
0 & \node (b12) {1}; & 0 & 0 \\
0 & 0 & 0 & 0 \\
0 & 0 & 0 & 1 \\
0 & \node (b42) {0}; & 0 & 0 \\
  };
    \draw [dashed,red,thick] (b12.north east) -- (b42.south east);
    \draw[thick, ->] (1,0) -- (2,0);
        \node at (1.5,.5) {$F_{LR}$};
    \node at (-2.8,0) {$,$};
        \matrix (mymatrix1) at (3.5,0) [matrix of math nodes,left delimiter={[},right
delimiter={]}]
  {
1 & 0 \\
0 & 0 \\  
0 & 1 \\
0 & 0 \\
  };
\end{tikzpicture}
\end{center}
Likewise, the following two matrices are folded onto $A_2$ by $F_{LR}$:
\begin{center}
\begin{tikzpicture}[node distance=-1ex]
\matrix (mymatrix1) at (-4.5,0) [matrix of math nodes,left delimiter={[},right
delimiter={]}]
  {
0 & \node (a12) {1}; & 0 & 0 \\
0 & 0 & 0 & 1 \\
0 & 0 & 0 & 0 \\
0 & \node (a42) {0}; & 0 & 0 \\
  };
  \draw [dashed,red,thick] (a12.north east) -- (a42.south east);
\matrix (mymatrix2) at (-1,0) [matrix of math nodes, left delimiter={[},right
delimiter={]}]
  {
0 & \node (b12) {0}; & 1 & 0 \\
0 & 0 & 0 & 1 \\
0 & 0 & 0 & 0 \\
0 & \node (b42) {0}; & 0 & 0 \\
  };
    \draw [dashed,red,thick] (b12.north east) -- (b42.south east);
    \draw[thick, ->] (1,0) -- (2,0);
        \node at (1.5,.5) {$F_{LR}$};
    \node at (-2.8,0) {$,$};
        \matrix (mymatrix1) at (3.5,0) [matrix of math nodes,left delimiter={[},right
delimiter={]}]
  {
1 & 0 \\
0 & 1 \\  
0 & 0 \\
0 & 0 \\
  };
\end{tikzpicture}
\end{center}
We find that, in total, there are four matrices that fold onto $J_2$. 
\end{Example}

We are now ready to present a formula for the number of elements of $\mc{B}_G$ that lie in the preimage of the folding map $F'$. 
\begin{Theorem}\label{T:formula}
The number of elements of rank $k$ in $\mc{B}_G$ is given by the formula
\begin{align}
\sum_{a+b+c =k} 2^{a+c}3^b   {l \choose b} S(l+1, l+1-a) S(l+1, l+1-c),
\end{align}
where $(a,b,c)\in \Z_{\geq 0}^3$.
\end{Theorem}
\begin{proof}
Let $A$ be an element from $\mc{R}_l$ with the triangular decomposition $A= A_l + A_d + A_u$. 
Then it is easy to verify that 
\begin{align}\label{A:threefactors}
|F'^{-1}(A)\cap \mathcal{B}_{Sp_n} | = 
|F'^{-1}(A_l)\cap \mathcal{B}_{Sp_n} ||F'^{-1}(A_d)\cap \mathcal{B}_{Sp_n} ||F'^{-1}(A_u)\cap \mathcal{B}_{Sp_n} |.
\end{align}
Let us denote the matrix ranks of $A_l,A_d$, and $A_u$ by $a,b$, and $c$, respectively. 
Then we denote the three factors on the right hand side of (\ref{A:threefactors}) by the notation
$f_a(A),f_b(A)$, and $f_u(A)$, respectively. 
Our choice of the subscripts for $f_a,f_b,f_c$ will be clarified in the next two paragraphs. 

As it was shown for the special case of $J_2$ in Example~\ref{E:indicates}, 
if $B$ is an element from $F'^{-1}(A)$, then the lower $l\times l$ half of the $2l\times l$ matrix 
$F_{LR}(B)$ must be an upper triangular matrix. 
In other words, when we unfold $A$ to a $2l\times l$ matrix, all of the nonzero entries of $A_l$ are 
moved into the upper $l\times l$ portion of the resulting matrix, hence, there is a unique $2l\times l$
matrix $A_l'$ such that $F_{TB}(A_l')= A_l$.  
Moreover, for every subset of the nonzero entries of $A'_l$, there exists a unique $A_l''$ in $\mc{B}_G$ 
such that $F_{LR}(A_l'') = A_l$. 
It follows from these arguments that 
\begin{align}\label{A:fa}
f_a (A) = | F^{-1}_{LR} (A_l')| = {a \choose 0} + {a \choose 1} + \cdots + {a \choose a} = 2^a.
\end{align}
A similar argument shows that 
\begin{align}\label{A:fc}
f_c (A) = {c \choose 0} + {c \choose 1} + \cdots + {c \choose c} = 2^c.
\end{align}

We now consider the possible unfolding of the diagonal matrix $A_d$. Recall that the rank of $A_d$ is $b$. 
For every $s$ element subset of the set of nonzero entries of $A_d$, there exists a unique $2l\times l$
matrix $A_d'$ such that $F_{TB}(A_d') =A_d$ and $\text{rank}(A_d') = s$.  
Likewise, for every $r$ element subset of the set of nonzero entries of $A_d'$, there exists a unique 
$2l\times 2l$ matrix $A_d''$ such that $F_{LR}(A_d'') = A_d'$ and $\text{rank}(A_d'') =r$. 
In total, there exist $\sum_{s=0}^b \sum_{r=0}^s {s\choose r} {b\choose s}$ elements in the preimage $F'^{-1}(A_d)$. 
But this double sum has a closed form:
\begin{align}\label{A:fb}
f_b(A) = \sum_{s=0}^b \sum_{r=0}^s {s\choose r} {b\choose s} = 3^b.
\end{align}
By combining (\ref{A:fa}), (\ref{A:fc}), and (\ref{A:fb}), we find that 
$|F'^{-1}(A)\cap \mathcal{B}_{Sp_n} | = 2^a 2^c 3^b = 2^{a+c}3^b$, 
which actually depends only on the ranks of the matrices $A_l,A_d$, and $A_c$. 
Our formula now follows from Proposition~\ref{P:triangular}.
\end{proof}

\begin{Remark}

It is easy to check that $\mc{B}_{Sp_n}(1)$ is equal to $\mc{B}_n(1)$. 
The Hasse diagram of $(\mc{B}_n(1),\leq)$ is a fishnet, see~\cite[Figure 1.9]{CanCherniavsky}.
By Theorem~\ref{T:inrsn}, we know that, for every $k\in \{1,\dots, l\}$, 
$(\mc{B}_{Sp_n}(k),\leq)$ is a subposet of $(\mc{B}_n(k),\leq)$.
However, unless $k=1$, the inclusion map $\mc{B}_{Sp_n}(k)\hookrightarrow \mc{B}_{n}(k)$ does not preserve the interval 
structure. Indeed, already for $k=2$, the cardinalities of $\mc{B}_{Sp_4}(2)$ and $\mc{B}_{4}(2)$ are different. 

\end{Remark}

\section{Rationally Smooth Borel Submonoids}\label{S:RationallySmooth}

Let $M$ be a reductive monoid with zero. Let $G$ denote its unit group, which is a connected reductive group. 
Then $M$ is called a {\em semisimple monoid} if $G$ has a one-dimensional center. 
The classification of semisimple (smooth) monoids is due to Renner~\cite{Renner85}. 
It turns out that a semisimple monoid $M$ is smooth if and only if $M$ is isomorphic to the monoid of $n\times n$ matrices,
$M_n$, for some $n\in \Z_+$. 
Note that the situation for general reductive monoids is not very different;
a reductive monoid with zero is smooth if and only if $M$ is of the form 
\[
M= \left(G_0 \times \prod_{i=1}^r M_{n_i}\right)/Z,
\]
where $Z$ is a finite central torus that does not intersect the unit group of $\prod_{i=1}^r M_{n_i}$, 
and $G_0$ is a semisimple subgroup~\cite[Section 11]{Timashev03}.
The semisimple monoids whose cohomological properties are as good as one hopes for are identified by Renner also~\cite{Renner08}.
They are called ``rationally smooth'' monoids.

Let $X$ be a complex algebraic variety with $\dim X = n$, and let $x$ be a point in $X$.
The variety $X$ is called {\em rationally smooth at $x$} if there exists an open neighborhood $U$ of $x$ 
such that for all $y\in U$, the following holds:
\[
H^m (X, X\setminus \{ y\}) = 
\begin{cases}
\{0\} & \text{ if } m\neq 2n;\\
\Q & \text{ if } m =2n.
\end{cases}
\]
$X$ is called {\em rationally smooth} if it is rationally smooth at every point $x$ in $X$. 
The classification as well as various characterizations of rationally smooth reductive monoids is given in~\cite{Renner08,Renner09}.

We will now adapt another result of Renner~\cite[Theorem 2.2]{Renner09} in our setting.

\begin{Lemma}\label{L:Renner's}
Let $X$ and $Y$ be two (normal) Borel submonoids of the (normal) reductive monoids $M$ and $N$, respectively.
Assume that both of $M$ and $N$ have zero elements, and assume that there exists a finite dominant morphism of 
algebraic monoids $g : M\to N$. Under these assumptions, $X$ is rationally smooth if and only if $Y$ is rationally smooth. 
\end{Lemma}
\begin{proof}
By abuse of notation, we will denote the restriction $g\vert_X$ by $g$ also. 
By our assumptions, the algebraic monoids $X$ and $Y$ have zero elements, denoted by $0_X$ and $0_Y$, respectively. 
Let $B$ denote the Borel subgroup contained in $X$, and let $T$ be a maximal torus contained in $B$. 
Then $0_X$ is the unique closed $B\times B$ orbit in $X$, hence, $X\setminus \{0\}$ is (rationally) smooth. 
Let $\mu : \C^* \times M \to M$ be a central (in $B$) one-parameter subgroup action on $M$ such that 
$\lim_{t\to 0} \mu(t,x) = 0_X$ for every $x\in M$. 
Then the quotients $\PP_X:=(X\setminus \{0_X\})/\C^*$ and $\PP_M := (M\setminus \{ 0_X \})/\C^*$ are 
projective $T\times T$ varieties such that $\PP_X \subsetneq \PP_M$. 
Furthermore, $\PP_X$ and $\PP_M$ are rationally smooth. 
Similarly, we have the rationally smooth, projective $T'\times T'$ varieties $\PP_Y\subsetneq \PP_N$, where $T'$ is the maximal torus in $g(T)\subset Y$.

By result of Brion~\cite[Lemma 1.3]{Brion99} we know that $X$ (resp. $Y$) is rationally smooth if and only if 
the Euler characteristic of $\PP_X$ (resp. the Euler characteristic of $\PP_Y$) is equal to the number of $T\times T$
fixed points in $\PP_X$ (resp. the number of $T'\times T'$ fixed points in $\PP_Y$).
Let us denote by $C(M)$ (resp. by $C(N)$) the closed $G\times G$-orbit in $\PP_M$ (resp. the closed $G'\times G'$-orbit in $\PP_N$),
where $G$ (resp. $G'$) is the unit group of $M$ (resp. of $N$). 
Since the $T\times T$ fixed points in $\PP_X$ lie in the closed intersection $\PP_X \cap C(M)$, 
and since $g |_{C(M)} : C(M)\to C(N)$ is a bijection, we see that the Euler characteristics of $X$ and $Y$ are equal.
In particular, $X$ is rationally smooth if and only if $Y$ is rationally smooth. 
\end{proof}

Two reductive monoids $M$ and $N$ are said to be {\em equivalent} if there exists a 
reductive monoid $L$ with two finite dominant morphisms $L\to M$ and $L\to N$.
If $M$ and $N$ are equivalent monoids, then we will write $M\sim_0 N$. It is easy to verify that $\sim_0$ is an equivalence relation. 
Let $M$ be a reductive monoid with zero. 
According to~\cite[Theorem 2.4]{Renner09}, 
\begin{align}\label{A:iff}
\text{$M$ is rationally smooth $\iff$ $M\sim_0 \prod_{i=1}^r M_{n_i}$.}
\end{align}

We are now ready to prove the main result of this section. 
\begin{Theorem}\label{T:rationallysmooth}
Let $M$ be a rationally smooth reductive monoid with zero. Let $B$ be a Borel subgroup in $M$ and let $\overline{B}$ denote 
the corresponding Borel submonoid. Then $\overline{B}$ is a rationally algebraic semigroup. 
\end{Theorem}

\begin{proof}
Since $M$ is rationally smooth, we know from (\ref{A:iff}) that there exists a reductive monoid $L$ admitting two finite dominant morphisms: 
\begin{center}
\begin{tikzpicture}
\node (a) at (0,1.5) {$L$};
\node at (-1,1.2) {$f$};
\node at (1,1.2) {$g$};
\node (b) at (-2,0) {$M$};
\node (c) at (2,0) {$\prod_{i=1}^r M_{n_i}$};
\draw[thick, ->] (a) -- (b);
\draw[thick, ->] (a) -- (c);
\end{tikzpicture}
\end{center}
Without loss of generality, we may assume that $L$ has a zero. 
Let $B_L$ be a Borel subgroup of $L$. As $f$ and $g$ are finite and dominant morphisms, they are surjective. 
In particular, the subgroups $f(B_L)$ and $g(B_L)$ are Borel subgroups in $M$ and $\prod_{i=1}^r M_{n_i}$, respectively. 
We set $X:= \overline{f(B_L)}$ and $Y:= \overline{g(B_L)}$. 
Then $X$ and $Y$ are Borel submonoids in $M$ and $\prod_{i=1}^r M_{n_i}$, respectively. 
Since $Y$ is (rationally) smooth, by Lemma~\ref{L:Renner's}, so is $\overline{B_L}$.
Once again by using Lemma~\ref{L:Renner's}, we see that $X$ is rationally smooth. 
This finishes the proof of our theorem. 
\end{proof}

As an application of Theorem~\ref{T:rationallysmooth}, we consider the symplectic monoid $MSp_n$. 
By the Renner's classification of rationally smooth simple group embeddings~\cite[Corollary 3.5]{RennerDescent}, 
$MSp_n$ is a rationally smooth semisimple monoid. 
Therefore, by Theorem~\ref{T:rationallysmooth}, its Borel submonoid is rationally smooth.

\section{Final Remarks: Nilpotent Subsemigroups}\label{S:Final}

In this section we will contrast some properties of the Borel submonoids of $M_n$ and $MSp_n$.
We begin with a general observation. 

\begin{Lemma}\label{L:tworequirements}
Let $M$ be a reductive monoid with the Bruhat-Chevalley-Renner decomposition $M = \sqcup_{r\in R} B\dot{r}B$,
where $B$ is a Borel subgroup in $M$, and $R$ is the Renner monoid of $M$. 
If an element $r$ from $R$ satisfies the following two properties, then every element of the orbit $B\dot{r}B$,
where $\dot{r}$ is a representative of $r$ in $\overline{N_G(T)}$, is nilpotent:
\begin{enumerate}
\item[(1)] $r$ is nilpotent in $R$, that is, $r^m=0$ for some $m\in \Z_+$;
\item[(2)] $r\lneq 1$.
\end{enumerate}
\end{Lemma}
\begin{proof}
Since $M$ is a linear algebraic monoid, it admits an embedding into $M_n$ for a sufficiently large positive integer $n$. 
By conjugating with an element of $GL_n$, we assume that $B$ is contained the upper triangular 
Borel submonoid of $M_n$. 
Clearly, if we can prove our assertion for $M= M_n$ and $B=B_n$, then the general case will follow. 
In this case, the Renner monoid is given by the rook monoid $\mc{R}_n$, 
and we can identify $\mc{R}_n$ as a submonoid of $M_n$. 
An element $r$ from $\mc{R}_n$ satisfies the two properties in our hypotheses if and only if it is a strictly upper triangular rook. 
But the product of an upper triangular matrix with a strictly upper triangular matrix is strictly upper triangular. 
Therefore, any element of $BrB$ is strictly upper triangular, hence, nilpotent. 
This finishes the proof of our assertion.
\end{proof}
We should note that we cannot replace any of the two requirements in Lemma~\ref{L:tworequirements}.
Indeed, if $r$ is not nilpotent, then any of its representatives $\dot{r}$, which is contained in $B\dot{r}B$, is not nilpotent. 
For the second item, we consider the matrices $r= \begin{bmatrix} 0 & 0 \\ 1 & 0 \end{bmatrix}$ and $b_2 = \begin{bmatrix} 
1 &  1 \\ 0 & 1\end{bmatrix}$. Then, we have $r b_2 = \begin{bmatrix} 0 &  0 \\ 1 & 1\end{bmatrix}$, which is not nilpotent.
Evidently, the set of nilpotent elements in a Borel submonoid is a closed subset. 

\begin{Definition}
The subvariety $\overline{B}^{nil} := \{  x\in \overline{B} :\ x^m = 0\text{ for some $m\in \Z_+$}\}$ will be called the {\em nilpotent semigroup}
of $\overline{B}$. 
\end{Definition}

\begin{Corollary}\label{C:subsemigroup}
Let $B$ be a Borel subgroup in a reductive monoid $M$ with zero. 
Then the nilpotent semigroup of $\overline{B}$ is a $B\times B$-stable algebraic subsemigroup of $M$.
\end{Corollary}

\begin{proof}
By definition, $\overline{B}^{nil}$ is defined by the polynomial relations $x^m=0$ ($m\in \Z_+$),
therefore, it is a closed subset of $\overline{B}$.  
By Lemma~\ref{L:tworequirements}, we know that $\overline{B}^{nil}$ is $B\times B$-stable. 
In fact, the proof of this lemma shows that $\overline{B}^{nil}$ is a semigroup. 
\end{proof}

Next, we will show that $\overline{B_n}^{nil}$ is an irreducible variety. 
\begin{Proposition}\label{P:num_nilpotent}
The nilpotent semigroup of $M_n$ is an irreducible algebraic semigroup of dimension ${n\choose 2}$. 
\end{Proposition}
\begin{proof}
By Lemma~\ref{L:tworequirements}, we know that $\overline{B_n}^{nil}= \bigsqcup B_n r B_n$, where the union is 
over all strictly upper triangular rooks in $\mc{R}_n$. It is easy to check that 
\begin{enumerate}
\item $r_0 := ( 0, 1,2,\dots, n-1)$ is a strictly upper triangular rook;
\item if $r$ is a strictly upper triangular rook, then $r\leq r_0$.
\end{enumerate}
These two conditions imply that $\overline{B_n r_0 B_n } = \overline{B_n}^{nil}$. 
It is easy to check that $\ell(r_0) = 1+ 2 + \cdots + (n-1) = {n\choose 2}$.
Since the orbit $B_n r_0 B_n$ is an irreducible variety, so is $\overline{B_n}^{nil}$. 
Thus, in light of Corollary~\ref{C:subsemigroup}, the proofs of our assertions are finished.
\end{proof}

Unfortunately the nice situation as in Proposition~\ref{P:num_nilpotent} does not hold for the nilpotent semigroup of $\overline{B_{Sp_n}}$. 
It turns out that $\overline{B_{Sp_n}}^{nil}$ has many irreducible components in varying dimensions.
See Figure~\ref{F:Figure2} for an example.
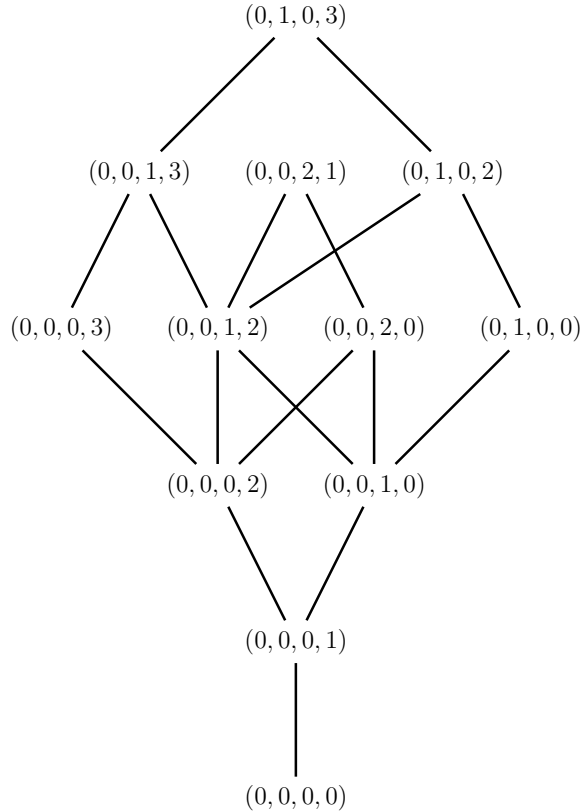
\begin{figure}[htp]
\begin{center}
\scalebox{.8}{
\begin{tikzpicture}[scale=.65]

\node at (0,20) (f3) {$(0,1,0,3)$};

\node at (-4,16) (e2) {$(0,0,1,3)$};
\node at (0,16) (e3) {$(0,0,2,1)$};
\node at (4,16) (e5) {$(0,1,0,2)$};

\node at (-6,12) (d1) {$(0,0,0,3)$};
\node at (-2,12) (d2) {$(0,0,1,2)$};
\node at (2,12) (d3) {$(0,0,2,0)$};
\node at (6,12) (d4) {$(0,1,0,0)$};

\node at (-2,8) (c1) {$(0,0,0,2)$};
\node at (2,8) (c2) {$(0,0,1,0)$};

\node at (0,4) (b1) {$(0,0,0,1)$};

\node at (0,0) (a) {$(0,0,0,0)$};

\draw[-, very thick] (a) to (b1);
\draw[-, very thick] (b1) to (c1);
\draw[-, very thick] (b1) to (c2);
\draw[-, very thick] (c1) to (d1);
\draw[-, very thick] (c1) to (d2);
\draw[-, very thick] (c1) to (d3);

\draw[-, very thick] (c2) to (d2);
\draw[-, very thick] (c2) to (d3);
\draw[-, very thick] (c2) to (d4);

\draw[-, very thick] (d1) to (e2);

\draw[-, very thick] (d2) to (e2);
\draw[-, very thick] (d2) to (e3);
\draw[-, very thick] (d2) to (e5);
\draw[-, very thick] (d3) to (e3);
\draw[-, very thick] (d4) to (e5);

\draw[-, very thick] (e2) to (f3);
\draw[-, very thick] (e5) to (f3);

\end{tikzpicture}
}
\end{center}
\caption{The Hasse diagram of $\overline{\mc{B}_{Sp_4}}^{nil}$.}
\label{F:Figure2}
\end{figure}

\bibliography{References}
\bibliographystyle{plain}

\end{document}